\documentclass[a4paper,11pt]{article}   
\usepackage{lineno,hyperref}
\usepackage[T2A]{fontenc}
\usepackage[utf8]{inputenc}
\usepackage[english]{babel}
\usepackage{ifthen}
\usepackage{amssymb,amsmath}
\usepackage[usenames]{color}
\date{}

 \newif\ifNoRemark
    \def\addtheorem#1#2#3#4{ 
    \ifthenelse{\expandafter\isundefined\csname the#2\endcsname}{\newcounter{#2}}{}
    \newenvironment{#1}[1][\global\NoRemarktrue]
     {\par\addvspace{2mm}\noindent
       \refstepcounter{#2}{\bf #3~\csname the#2\endcsname
      \vphantom{##1}\ifNoRemark.\ \else\ (##1).\fi}\begingroup #4}%
     {\endgroup\par\addvspace{1mm}\global\NoRemarkfalse}
    \expandafter\newcommand\csname b#1\endcsname{\begin{#1}}
    \expandafter\newcommand\csname e#1\endcsname{\end{#1}}
    }

\newtheorem{theorem}{Theorem}
\newtheorem{lemma}{Lemma}
\newtheorem{proposition}{Proposition}
\newtheorem{corollary}{Corollary}

\def\bproof{\textit{Proof.}\ }
\def\proofend{$\square$\vspace{0.3em}\par}

\begin{document}


\title{On extremal properties of perfect $2$-colorings
}

\author{Vladimir N. Potapov\\
 Sobolev Institute of Mathematics\\
{vpotapov@math.nsc.ru}}

\maketitle

\begin{abstract}
A coloring of the vertices of a graph is called perfect if, for
every vertex, the collection of colors of its neighbors depends only
on its own color. The corresponding partition of the vertex set is
called equitable. We note that a collection of bounds (Hoffman
bound, Haemers bound, Cheeger bound, Bierbrauer--Friedman  bound,
etc) is only reached on perfect $2$-colorings.  We show that the
Expander Mixing Lemma is another example of an inequality that is
related to a perfect $2$-coloring. For an amply regular graph
$G=(V,G)$, we prove a new upper bound for the size of a subset
$S\subset V$ with the fixed average internal degree. This bound is
reached on a set $S$ if and only if $\{ S, V\setminus S\}$ is an
equitable partition. We improve the Hoffman bound in a special case.
\end{abstract}

Keywords: equitable partition, perfect coloring, amply regular
graph, tight bound, expander mixing lemma, sensitivity of Boolean
function

 MSC[2020] 05C35, 05B30, 05E30



\section{Introduction}

The concept of a perfect coloring of a graph arises independently in
graph theory, algebraic combinatorics, cryptography, and coding
theory. A coloring of vertices of a graph is called perfect if, for
every vertex, the collection of colors of its neighbors depends only
on its own color. Other terms used for this notion in the literature
are ``equitable partition'', ``partition design'', ``distributive
coloring'', ``intriguing set'', ``extremal graphical designs'',  and
``regular boolean function''. The last three terms are  used only
for perfect $2$-colorings. Moreover,  perfect $2$-colorings of a
regular graph are the same as the completely regular codes  with
covering radius $1$. A comprehensive survey of the theory of perfect
colorings and related topics is available in \cite{CRC}.

It is very useful to formulate the definition of a perfect coloring
in  terms of linear algebra (see \cite{Martin}). A surjective
function $f:V\to \{0,\dots,k-1\}$ is a perfect $k$-coloring of a
graph $\Gamma=(V,E)$ with quotient  matrix~$Q$ of size $k\times k$
if and only if
\begin{equation}\label{1eq:AFFS}
 M\cdot F=F\cdot Q,
\end{equation}
where $M$ is the $|V|\times|V|$ adjacency matrix of~$\Gamma$ and $F$
is the incidence   $\{0,1\}$-matrix of size $|V|\times k$ for~$f$ .
If we admit rational matrices $F$, then each $F$ satisfying equation
(\ref{1eq:AFFS}) is called a perfect structure (see \cite{ATar}).

Some years ago, Avgustinovich \cite{Avg} postulated the following
principle. Each perfect coloring is a solution of some optimization
problem and on the other hand, a tight (reached a theoretical bound)
solution of any optimization problem in combinatorics is equivalent
to a perfect coloring with some parameters that depend on the
problem. The first part of this principle is not difficult. Such a
problem can be formulated as Proposition \ref{Hamming} or as
\cite[Lemma 3]{Krotov12}. Previously, we collected many examples
illustrating  the second part of this principle in \cite{PotAv}. In
this paper, we consider a number of other examples.

Indeed, perfect $2$-colorings turn out to be the tight solutions of
some extremal problems. If we consider the following bounds for the
number of vertices, then the set of vertices attaining one of these
bounds must be a perfect $2$-coloring. This property holds in the
cases of the Hamming and Singleton bounds in coding theory, the
Hoffman bound on independent sets \cite{Hoff}, Cheeger's bound on
cut sizes \cite{AlonCh}, \cite{Golubev}, Haemers' bounds on the
subgraph degree \cite{Haemers}, the Fon-der-Flaass \cite{FdF2} and
Bierbrauer--Friedman bounds on orthogonal arrays (see a proof of the
Bierbrauer--Friedman bound in \cite{Bier}, \cite{Friedman} and a
proof of the property to be a perfect $2$-coloring in \cite{OPPh},
\cite{Pot12}). Binary orthogonal arrays attaining the other bound
are related to perfect $3$-colorings~\cite[Theorem~1]{Krotov1536}.
In this paper we prove that the Expander Mixing Lemma is another
example of an inequality whose attainment implies a perfect
$2$-coloring (Lemma~\ref{EMLclaim}).

 It is possible to consider $q$-ary
Boolean-valued function in $k$ variables as a $2$-coloring of the
Hamming graph $H(k,q)$. Then the sensitivity of the function is
equal to the number of mixed-colored edges in the graph. We prove
that some known bounds on the sensitivity follow from the Expander
Mixing Lemma. Moreover, they are only reached on perfect
$2$-colorings (Corollaries \ref{cut12}, \ref{cut13}). A bound for a
partition into several subsets of vertices with the fixed average
internal degree was proved by Krotov \cite[Sect.~3]{Krotov12}. This
bound is tight and it is only reached on perfect colorings. We note
that the tight solutions of extremal problems are a rich source of
combinatorial designs. The above connection can be useful for
optimization theory and theory combinatorial structures.

Distance-regular graphs are one of the main topic of algebraic graph
theory. There are many theorems for distance-regular graphs that are
impossible for arbitrary regular graphs (see, e.g., \cite{BCN}).
Amply regular graphs (regular up to distance $2$) occupy an
intermediate position between regular and distance-regular graphs.
Thus,  some stronger results can be obtained for them. We prove a
new upper bound for the cardinality of a subset of vertices of an
amply regular graph with the fixed average internal degree
(Theorem~\ref{thEPPC}). The bound is tight and attained only  on
perfect $2$-colorings. Moreover, any perfect $2$-coloring of an
amply regular graph attains this bound. This bound is the better
than the Hoffman bound in some cases (Corollary~\ref{corExt2}).

\section{Preliminaries}

Next we formulate a combinatorial definition of a perfect coloring.
Let $G=(V,E)$ be a regular graph, where $V$ is the set of vertices
and $E$ the set of edges. Throughout the article we denote by
$n=|V|$ the number of vertices of $G$ and denote by $M$ the
adjacency matrix of $G$. A function $f:V\rightarrow I$ is called a
perfect coloring if there are integers $q_{ij}$, $i,j\in I$, such
that every vertex of $C_i=f^{-1}(i)$ is adjacent to $q_{ij}$
vertices of $C_j=f^{-1}(j)$. In this case the partition $\{C_i :
i\in I\}$ is called equitable. We will use the both equivalent terms
``perfect coloring'' and ``equitable partition''.

It is well known that every eigenvalue of a perfect coloring
(i.\,e., an eigenvalue of the quotient matrix $Q=(q_{ij})$) is an
eigenvalue of the adjacency matrix~$M$. For an $r$-regular graph,
the largest (both as a signed value and in absolute value)
eigenvalue is equal to~$r$ and it coincides with the largest
eigenvalue of the quotient matrix of any perfect coloring.

The following connection between perfect $2$-colorings and
eigenfunctions of a graph is known.

\begin{proposition}[\cite{FdF2}]\label{corol5}
A two valued function $f:V\rightarrow \mathbb{R}$ is a perfect
$2$-coloring of a regular connected graph $G(V,E)$ if and only if
there exists a constant~$\gamma$ such that $f-\gamma{\mathbf 1}_V$
is an eigenfunction of~$G$ with eigenvalue $\lambda$, where
$\lambda$ is an eigenvalue of the corresponding quotient matrix.
\end{proposition}

Here ${\mathbf 1}_V$ is the indicator function of $V$. The following
two  propositions is well known. We prove it for the completeness.

\begin{proposition}\label{propEP}
Let $G=(V,E)$ be an $r$-regular connected graph with $n$ vertices.
Suppose a partition $\{A, B\}$, $ B=V\setminus A$, is equitable,
i.\,e., for some $\gamma>0$ the function
$f={\mathbf{1}}_A-\gamma{\mathbf{1}}_V$ is an eigenfunction with an
eigenvalue $\lambda\neq r$. Then the
 quotient matrix of the equitable partition is equal
to $\begin{pmatrix}
\frac{r|A|+\lambda|B|}{n} & \frac{(r-\lambda)|B|}{n}\\
\frac{(r-\lambda)|A|}{n} & \frac{r|B|+\lambda|A|}{n}
\end{pmatrix}$.
\end{proposition}
\bproof Without loss of generality, the quotient matrix of an
equitable partition is
  $Q=\begin{pmatrix}
r-b & b \\
c& r-c
\end{pmatrix}$ for some~$b$ and~$c$.
By double counting  the number of edges that connect elements of~$A$
with elements of~$B$, we obtain $b|A|=c|B|$. By Proposition
\ref{corol5} the second eigenvalue of~$Q$ equals~$\lambda=r-b-c$.
Therefore, $b=(r-\lambda)/(1+\frac{|A|}{|B|})$ and
$c=(r-\lambda)/(1+\frac{|B|}{|A|})$. Since $n=|A|+|B|$, the
proposition is proven. \proofend

The simplest extremal property of perfect $2$-colorings is the
following.

\begin{proposition}\label{Hamming} Let $G=(V,E)$ be an $r$-regular graph and let
$S\subset V$.
For a vertex~$x$ from~$S$, let~$a(x)$ denote the number of neighbors
of~$x$ in~$S$. We suppose $a(x)\leq a$ for some constant~$a$. For a
vertex~$y$ from~$V\setminus S$, let~$d(y)$ denote the number of
neighbors of~$y$ in~$V\setminus S$. We suppose $d(y)\geq d$ for some
constant~$d$. Then it holds
$$\frac{|S|}{n}\leq \frac{r-d}{2r-a-d}.$$ Moreover, in the case of
equality, ${\mathbf{1}}_S$ is a perfect $2$-coloring with the
quotient matrix
  $\begin{pmatrix}
a & r-a \\
r-d & d
\end{pmatrix}$.
\end{proposition}
\bproof Let $M$ be the adjacency matrix of~$G$. It is easy to see
that
\begin{equation}\label{PCe10}
(M{\mathbf{1}}_S,{\mathbf{1}}_S)\leq a|S|,\qquad   (M{\mathbf{1}}_S,
\mathbf{1}_V - {\mathbf{1}}_S)=({\mathbf{1}}_S, M(\mathbf{1}_V -
{\mathbf{1}}_S)) \leq (r-d)(n-|S|)
\end{equation}
 and $(M{\mathbf{1}}_S,\mathbf{1}_V)=r|S|$. Suppose that there exists
 $x\in S$ such that $a(x)< a$ or  there exists  $y\in V\setminus S$ such that
 $d(y)>d$. Then one of   inequalities~(\ref{PCe10}) should be strict.

 It holds  $
(M\mathbf{1}_S,\mathbf{1}_V)=(M\mathbf{1}_S, \mathbf{1}_V -
\mathbf{1}_S)+(M\mathbf{1}_S,\mathbf{1}_S)$.  Therefore, $r|S|\leq
(r-d)(n-|S|)+a|S|$, i.\,e., $\frac{|S|}{n}\leq \frac{r-d}{2r-d-a}$.
 If one of inequalities
(\ref{PCe10}) is strict, then $r|S|< (r-d)(n-|S|)+a|S|$, i.\,e.,
$\frac{|S|}{n}< \frac{r-d}{2r-d-a}$. Thus, in the case of the
equality we obtain  $a(x)= a$ for each $x\in S$ and $d(y)=d$ for
every $y\in V\setminus S$. \proofend

A vertex subset $S$ in an $r$-regular graph is called a $1$-perfect
code if the partition $\{S,V\setminus S\}$ is an equitable with
quotient matrix
  $\begin{pmatrix}
0 & r \\
1& r-1
\end{pmatrix}$. In
the case $a=0$ and $d=r-1$, Proposition~\ref{Hamming} corresponds to
the Hamming bound and it is a routine criterion
 for $1$-perfect codes.

Denote by $\lambda_{\mathrm{min}}$ the minimum eigenvalue of $M$.
The Hoffman (or Delsarte--Hoffman) upper bound~\cite{Hoff}
 on the cardinality of an
independent set in an $r$-regular graph~$G$   is equal to
$\frac{-\lambda_{\mathrm{min}}n|}{r-\lambda_{\mathrm{min}}}$. It is
well known (see, e.g.,~\cite{Golubev}) that if an independent set
$S$ attains the Hoffman bound, then partition $\{S, V\backslash S\}$
is  equitable. There exists a generalization of this fact to non
independent sets. Denote by~$\sigma(S)$ the average internal degree
for a set $S\subset V$, i.e., $\sigma(S)=(M{\mathbf{1}}_S,{
{\mathbf{1}}}_S)/|S|$.

\begin{theorem}[\cite{PotAv}]\label{thPotAv} Let $G=(V,E)$ be an
$r$-regular graph and let $S\subset V$. If  $\sigma(S)\leq a$, then
$|S|\leq
\frac{(a-\lambda_{\mathrm{min}})n}{r-\lambda_{\mathrm{min}}}$.
Moreover, $|S|=
\frac{(a-\lambda_{\mathrm{min}})n}{r-\lambda_{\mathrm{min}}}$ if and
only if ${\mathbf{1}}_S$
is a perfect $2$-coloring with quotient  matrix\\
$\begin{pmatrix}
a& r-a \\
a-\lambda_{\mathrm{min}} & r+\lambda_{\mathrm{min}}-a
\end{pmatrix}$.
\end{theorem}

Denote by $\chi(G)$ the chromatic number of a graph $G$. For an
$r$-regular graph~$G$, the inequality $\chi(G)\geq
\frac{\lambda_{\mathrm{min}}-r}{\lambda_{\mathrm{min}}}$ is a
corollary of the Hoffman bound.
\begin{corollary}
Let $G=(V,E)$ be an $r$-regular graph. If
$k=\frac{\lambda_{\mathrm{min}}-r}{\lambda_{\mathrm{min}}}$, then
every proper  $k$-coloring of~$G$ is a perfect $k$-coloring.
\end{corollary}
\bproof By Theorem \ref{thPotAv}, the partition $\{C_i, V\setminus
C_i\}$ is  equitable, where  $C_i$ is the set of $i$-colored
vertices for $i=1,\dots,k$. Therefore, every vertex from $V\setminus
C_i$ has the same number of adjacent $i$-colored vertices. In
particular, every vertex from~$C_j$ has the same number of adjacent
$i$-colored vertices for each~$i\neq j$. It remains to note that a
vertex from $C_i$ has no $i$-colored neighbors, because the coloring
is proper. \proofend

\section{Expander Mixing Lemma and its corollaries}

Let $A$ and $B$ be arbitrary (not necessarily disjoint) nonempty
subsets of $V$. Consider the set $\{(a,b):\ a\in A,\  b\in B,\
\{a,b\}\in E(G)\}$ of all arcs that connect vertices from~$A$ to
vertices from~$B$. The cardinality of this set is denoted
by~$e(A,B)$. In particular, if $a,b\in A\cap B$, then the edge
$\{a,b\}$ is counted twice, as the arcs~$(a,b)$ and~$(b,a)$. It is
easy to see that
\begin{equation}\label{EMLeq}
e(A,B)=({\mathbf{1}}_{A}, M{\mathbf{1}}_{B}),
 \end{equation} where~$M$ is the
adjacency matrix of~$G$.

The Expander Mixing Lemma is proven in~\cite{AlonCh} and it appears
in a form appropriate  for us, for example, in~\cite{Dev} (Lemma 3).
Moreover, a proof of an improvement of the lemma can be found
in~\cite{Dev}. Below, we establish that subsets attaining bound
(\ref{EML}) correspond to a perfect $2$-coloring. To prove this, we
need to repeat the proof of the Expander Mixing Lemma.

\begin{lemma}[Expander Mixing Lemma]\label{EMLclaim}
Let  $G$ be an $r$-regular connected graph and let $\lambda$ be the
second largest, in absolute value, eigenvalue of~$G$ (if $G$ is
bipartite, then $\lambda = -r$). Then
\begin{equation}\label{EML}
     \left|e(A,B)-\frac{r|A||B|}{n}\right|\leq
     |\lambda|\sqrt{|A|\,|B| \left(1-\frac{|A|}{n} \right) \left(1-\frac{|B|}{n} \right)}.
\end{equation}
      Moreover,
      this bound is reached
     if and only if  $B=V\setminus A$ or $B=A$  and the partition $\{A,V\setminus
     A\}$ is equitable  with the eigenvalue $\lambda$.
     \end{lemma}
\bproof Let $\lambda_k<\lambda_{k-1}<\dots<\lambda_{1}<\lambda_0=r$
be the eigenvalues of~$G$. By definition $\lambda=\lambda_k$ or
$\lambda=\lambda_1$. Since $M$ is a symmetric matrix, the direct sum
of the eigenspaces of all of $M$'s eigenvalues is the entire vector
space. Consider the indicator functions~${\mathbf{1}}_{A}$
and~${{\mathbf{1}}}_{B}$ as linear combinations of eigenfunctions
of~$G$, i.\,e., ${\mathbf{1}}_{A}=\sum_i\alpha_i\phi_i$ and
${{\mathbf{1}}}_{B}=\sum_i\beta_i\phi'_i$, where~$\phi_i$
and~$\phi'_i$ correspond to the eigenvalue $\lambda_i$. Note that
$\phi_i$ and~$\phi'_i$ may be different. Without loss of generality,
we require that $\|\phi_i\|_2=1$ for all~$i$. We assume
$\lambda_0=r$, so the corresponding eigenfunction is
$\phi_0=\mathbf{1}/\sqrt{n}$.
 Therefore
$\alpha_0=({\mathbf{1}}_{A},\phi_0)=|A|/\sqrt{n}$ and $\beta_0=({
{\mathbf{1}}}_{B},\phi_0)=|B|/\sqrt{n}$.  Then
\begin{equation}\label{eqEx1}
({\mathbf{1}}_{A}-\alpha_0\phi_0,
M({\mathbf{1}}_{B}-\beta_0\phi_0))= \sum\limits_{i\neq
0}\lambda_i\alpha_i\beta_i(\phi_i,\phi'_i).
\end{equation}
By the Cauchy-Schwarz inequality, we obtain
\begin{equation}\label{eqEx2}
\Big|\sum\limits_{i\neq
0}\lambda_i\alpha_i\beta_i(\phi_i,\phi'_i)\Big|\leq
|\lambda|\sum\limits_{i\neq 0}|\alpha_i\beta_i|\leq
|\lambda|\bigg(\sum\limits_{i\neq 0}\alpha^2_i
\bigg)^{1/2}\bigg(\sum\limits_{i\neq 0}\beta^2_i \bigg)^{1/2}.
\end{equation}
By~(\ref{EMLeq}), it holds
\begin{multline}\label{eqEx3}
({\mathbf{1}}_{A}-\alpha_0\phi_0,
M({\mathbf{1}}_{B}-\beta_0\phi_0))=
\\=({\mathbf{1}}_{A}, M{\mathbf{1}}_{B})- \beta_0({\mathbf{1}}_{A}-\alpha_0\phi_0,
M\phi_0)-\alpha_0(M\phi_0, {\mathbf{1}}_{B}-\beta_0\phi_0)-
(\alpha_0\phi_0, M(\beta_0\phi_0)) \\
= e(A,B)-\frac{r|A||B|}{n}.
\end{multline}
From $({\mathbf{1}}_{A},{\mathbf{1}}_{A})= \sum_i\alpha^2_i$, we
derive
\begin{equation}\label{eqEx31}
\sum\limits_{i\neq 0}\alpha^2_i= ({\mathbf{1}}_{A},{
{\mathbf{1}}}_{A})-\alpha^2_0=|A|- |A|^2/n.
\end{equation}
A similar equation takes place for $B$. Summarizing
(\ref{eqEx1})--(\ref{eqEx3}), we get
\begin{multline*}
 \left|e(A,B)-\frac{r|A||B|}{n}\right|=
|({\mathbf{1}}_{A}-\alpha_0\phi_0,
M({\mathbf{1}}_{B}-\beta_0\phi_0))| \leq |\lambda|\sqrt{(|A|-
|A|^2/n)(|B|- |B|^2/n)}.
\end{multline*}

In order to obtain the equality in the Cauchy-Schwarz inequality, it
is necessary vector ${\mathbf{1}}_{A}-\alpha_0\phi_0$ to be
collinear to the vector ${\mathbf{1}}_{B}-\beta_0\phi_0$. So, we
obtain that  $\phi_i=\phi'_i$ and $\alpha_i=\beta_i$ or $
\alpha_i=-\beta_i$  for all~$i\neq 0$. Moreover, in order to obtain
equality in the first inequality of~(\ref{eqEx2}), it is necessary
the function ${\mathbf{1}}_{B}-\beta_0\phi_0$ to be be an
eigenfunction with the eigenvalue~$\pm \lambda$. By
Proposition~\ref{corol5}, we see that ${\mathbf{1}}_{B}$ is a
perfect $2$-coloring. If the vector
${\mathbf{1}}_{A}-\alpha_0\phi_0$ has the same direction as ${
{\mathbf{1}}}_{B}-\beta_0\phi_0$, then $B=A$. If
${\mathbf{1}}_{A}-\alpha_0\phi_0$ has the  direction opposite  to
${\mathbf{1}}_{B}-\beta_0\phi_0$, then $A=V\setminus B$.

Let $\{A,V\setminus     A\}$ be an equitable partition with the
eigenvalue $\lambda$. Consider the case $B=V\setminus A$. We have
$1-\frac{|A|}{n}=\frac{|B|}{n}$ and
$|A|-\frac{|A|^2}{n}=\frac{|A||B|}{n}=|B|-\frac{|B|^2}{n}$. From the
quotient matrix of the equitable partition (see
Proposition~\ref{propEP}), we derive
$e(A,B)=\frac{(r-\lambda)|A||B|}{n}$. The equality in~(\ref{EML}) is
now straightforward. The second case $A=B$ is similar. \proofend

As stated in \cite{Dev} inequality (\ref{EML}) is one of the form of
Cheeger's bound. Let $A\subset V$ and $B=V\setminus A$. The set of
all edges that connect $A$ and $B$ is called a cut-set. The
cardinality $e(A,B)$ of the cut-set is called the cut size.

\begin{corollary}[\cite{Golubev}]\label{cut}
Let $G$ be an $r$-regular connected graph and let
$\lambda_k<\lambda_{k-1}<\dots<\lambda_{1}<\lambda_0=r$ be the
eigenvalues of~$G$. Then for any nonempty $A\subset V$ and
$B=V\setminus A$, it holds
$$ \frac{(r-\lambda_1)|A||B|}{n}\leq e(A,B)\leq \frac{(r-\lambda_k)|A||B|}{n}.$$
 Moreover, one of
this two bounds is reached
     if and only if  $\{A,B\}$ is an equitable partition with the eigenvalue $\lambda_1$
     or $\lambda_k$ respectively.
\end{corollary}
\bproof We will use the notation from the proof of Lemma
\ref{EMLclaim}. Since
${\mathbf{1}}_{B}={\mathbf{1}}_{V}-{\mathbf{1}}_{A}={\mathbf{1}}_{V}-\sum_i\alpha_i\phi_i$,
(\ref{eqEx1}) is equivalent to the equation
$$({\mathbf{1}}_{A}-\alpha_0\phi_0, M({\mathbf{1}}_{B}-\beta_0\phi_0))=
-\sum\limits_{i\neq 0}\lambda_i\alpha_i^2.$$  By the hypothesis of
the corollary, it holds
$$-\lambda_1\sum\limits_{i\neq 0}\alpha_i^2\leq -\sum\limits_{i\neq
0}\lambda_i\alpha_i^2\leq -\lambda_k\sum\limits_{i\neq
0}\alpha_i^2.$$ Utilizing (\ref{eqEx3}) and (\ref{eqEx31}), we
obtain the required inequality. The left (or right) side of the
inequality holds with equality if and only if
${\mathbf{1}}_{A}-\alpha_0\phi_0$ is an eigenfunction of
 $G$ with eigenvalue
$\lambda_1$ (or $\lambda_k$ respectively). In these cases, the
partition $\{A,B\}$ is  equitable  by Proposition \ref{corol5}.
\proofend

By Corollary \ref{cut}, we obtain that the maximum cut size
$(r-\lambda_k)n/4$ corresponds to an equitable partition with the
minimum eigenvalue.

As mentioned above an indicator functions~ of each subset $C\subset
V$ is the linear combination ${\mathbf{1}}_{C}=\sum_i\phi_i$, where
$\phi_i$ is an eigenfunction of~$G$ with eigenvalue $\lambda_i$ and
$\lambda_k<\lambda_{k-1}<\dots<\lambda_{1}<\lambda_0=r$. In an
arbitrary regular graph $G = (V,E)$, a nonempty set $C$ of vertices
is called an {\it algebraic T-design} if $\phi_i=0$ as $i\in T$ in
the this decomposition (see \cite{Delsarte}).

\begin{corollary}\label{cut11}
Let $G$ be an $r$-regular connected graph and let
$\lambda_k<\lambda_{k-1}<\dots<\lambda_{1}<\lambda_0=r$ be the
eigenvalues of~$G$.

$(a)$ Let $T=\{j+1,\dots,k\}$. Then for any nonempty T-design
$C\subset V$, it holds
$$ e(C,V\setminus C)\leq \frac{(r-\lambda_j)|C|(n-|C|)}{n}.$$

$(b)$ Let $T=\{1,\dots,j-1\}$. Then for any nonempty T-design
$C\subset V$, it holds
$$ e(C,V\setminus C)\geq \frac{(r-\lambda_j)|C|(n-|C|)}{n}.$$

 Moreover, any of
this two bounds is reached
     if and only if  $\{C,V\setminus C\}$ is an equitable partition with the eigenvalue $\lambda_j$.
\end{corollary}

A proof of Corollary \ref{cut11} is similar to the proof of
Corollary \ref{cut}.

Let $G$ be the Hamming graph $H(k,q)$. It is well known that in this
case $\lambda_{i}=k(q-1)-iq$. We can consider $f={\mathbf{1}}_{C}$
as a $q$-ary Boolean-valued function in $k$ variables. If $C$ is an
algebraic T-design and $T=\{1,\dots,j\}$, then $f$ is a
correlation-immune function of order $j$ (see, e.g., \cite{Pot12},
Proposition 2). If $C$ is an algebraic T-design and
$T=\{j+1,\dots,k\}$, then $f$ has degree $j$ (see, e.g.,
\cite{Val24}). In the theory of Boolean functions the value
$e(C,V\setminus C)/n$ is called the {\it average sensitivity} of $f$
and  is denoted by $I(f)$. The value $|C|/n$ is denoted by
$\rho(f)$.

By Corollary \ref{cut11} we immediately obtain the following
statements.

\begin{corollary}[\cite{Pot12}, Theorem 1]\label{cut12}
Let $f$ be a $q$-ary Boolean-valued function in $k$ variables, and
let ${\rm cor}(f)$ be the maximal order of its correlation immunity.
Then it holds
\begin{equation}\label{eqEx90}
I(f)\geq q({\rm cor}(f)+1)\rho(f)(1-\rho(f)).
\end{equation}
 Moreover,  this bound is reached
     if and only if $f$ is a perfect $2$-coloring.
\end{corollary}

Note that (\ref{eqEx90}) coincides with the Bierbrauer bound
\cite{Bier} if $C$ is an independent set. Indeed, in this case
$I(f)=\rho(f)k(q-1)$, so $1-\rho(f)\leq \frac{k(q-1)}{q({\rm
cor}(f)+1)}$.

\begin{corollary}\label{cut13}
Let $f$ be a $q$-ary Boolean-valued function in $k$ variables with
degree $d$. Then it holds
$$ I(f)\leq qd\rho(f)(1-\rho(f)).$$
Moreover,  this bound is reached
     if and only if $f$ is a perfect $2$-coloring.
\end{corollary}

Corollary \ref{cut13} is similar to the result by Valyuzhenich
(\cite{Val24}, Corollary 1) in the case $\rho(f)=\frac12$.

 By the definition,  $e(C,C)$ is the
doubled number of pairs of adjacent vertices in
 $C\subset V$. The following inequalities are well known. We
 establish only the connection with equitable partitions.

\begin{corollary}[\cite{Haemers}]\label{haemers}
Let $G$ be an $r$-regular connected graph, let $C\subset V$, and let
$\lambda_k<\lambda_{k-1}<\dots<\lambda_{1}<\lambda_0=r$ be the
eigenvalues of~$G$.   Then it holds
$$ \lambda_k|C|+\frac{(r-\lambda_k)|C|^2}{n}\leq e(C,C)\leq
 \lambda_1|C|+\frac{(r-\lambda_1)|C|^2}{n}.$$
 Moreover, one of this
two bounds is reached
     if and only if  $\{C,V\setminus C\}$ is an equitable partition
      with the eigenvalue $\lambda_1$ or $\lambda_k$ respectively.
\end{corollary}
\bproof We will use the notation from the proof of Lemma
\ref{EMLclaim}. Put
 $B=A=C$. In this case,
(\ref{eqEx1}) is equivalent to the equation
$$({\mathbf{1}}_{C}-\alpha_0\phi_0, M({\mathbf{1}}_{C}-\alpha_0\phi_0))=
\sum\limits_{i\neq 0}\lambda_i\alpha_i^2.$$ By the hypothesis of the
corollary, it holds
$$\lambda_k\sum\limits_{i\neq 0}\alpha_i^2\leq \sum\limits_{i\neq
0}\lambda_i\alpha_i^2\leq \lambda_1\sum\limits_{i\neq
0}\alpha_i^2.$$ Utilizing (\ref{eqEx3}) and (\ref{eqEx31}), we
obtain the required inequality. The left (or right) side of the
inequality holds with  equality if and only if
${\mathbf{1}}_{C}-c\phi_0$ is an eigenfunction of
 $G$ with eigenvalue
$\lambda_k$ (or $\lambda_1$ respectively). In these cases, the
partition $\{C,V\setminus C\}$ is  equitable  by Proposition
\ref{corol5}. \proofend

\section[Perfect 2-colorings of amply regular graphs]{Perfect $2$-colorings of amply regular graphs}

A graph $G$ is called  amply regular  if the distance-$2$ adjacency
matrix $M_2(G)$ is a polynomial
\begin{equation}\label{e1.2}
M_2(G)=p(M)=p_2M^2+p_1M+p_0I
\end{equation}
 on the adjacency matrix $M$. It is easy to see that
any  amply regular graph is an $r_p$-regular, where $r_p=-p_0/p_2$.
Denote by $\sigma_2(S)$ the average number of vertices at
distance~$2$ in the set $S\subset V$, i.e.,
$\sigma_2(S)=(M_2(G){\mathbf{1}}_S,{\mathbf{1}}_S)/|S|$. Recall that
$\sigma(S)=(M{\mathbf{1}}_S,{\mathbf{1}}_S)/|S|$.

The following theorem  is true for any number of elements in an
equitable partition. We formulate the case of two elements in the
partition, which is  sufficient for our objectives.

\begin{theorem}[\cite{Krotov12}]\label{krotov} Let $G$ be an amply regular graph with
polynomial~$p$ and  let $S\subset V$. If $\sigma(S)= a$ and
$\sigma(V\backslash S)= d $, then $\sigma_2(S)\leq (p(Q))_{11}$ and
$\sigma_2(V\backslash S)\leq (p(Q))_{22} $, where $Q=\begin{pmatrix}
a & r_p-a \\
r_p-d & d
\end{pmatrix}$.
Moreover,  the inequalities both hold with equality if and only if
${\mathbf{1}}_S$ is a perfect $2$-coloring with quotient matrix~$Q$.
\end{theorem}

By Theorem \ref{krotov}, we can easy obtain the following criterium
for perfect $2$-colorings with the minimum eigenvalue in amply
graphs.

\begin{corollary}\label{corExt}
Let $G$ be an amply $r$-regular graph with polynomial~$p$  and let
$S\subset V$ be an independent set.   Then $\sigma_2(S)\leq
-p_2r(\lambda_{\mathrm{min}}+1)$, where~$\lambda_{\mathrm{min}}$ is
the minimum eigenvalue of~$G$. Moreover, $\sigma_2(S)=
-p_2r(\lambda_{\mathrm{min}}+1)$ if and only if ${\mathbf{1}}_S$ is
a perfect $2$-coloring with the eigenvalue~$\lambda_{\mathrm{min}}$.
\end{corollary}
\bproof It is clear that  $\sigma(S)=0$ and $\sigma(V\backslash S)=
r(1-\frac{|S|}{n-|S|})$ for each independent set~$S$. It is easy to
find that $(p(Q))_{11}=p_2(r(r-d)-r)$. By Theorem~\ref{krotov}, it
holds $\sigma_2(S)\leq p_2(\frac{r^2|S|}{n-|S|}-r)$. By the Hoffman
bound, we get $\frac{r^2|S|}{n-|S|}\leq -\lambda_{\mathrm{min}}r$.
Thus $\sigma_2(S)\leq -p_2r(\lambda_{\mathrm{min}}+1)$. If
$\sigma_2(S)= -p_2r(\lambda_{\mathrm{min}}+1)$ then ${\mathbf{1}}_S$
is a perfect $2$-coloring by Theorem \ref{thPotAv}. For any perfect
$2$-coloring with the eigenvalue~$\lambda_{\mathrm{min}}$, the
equality $\sigma_2(S)= -p_2r(\lambda_{\mathrm{min}}+1)$ is
straightforward. \proofend

Suppose that the adjacency matrix $M$ satisfies (\ref{e1.2}). Then
by the definition of a perfect $2$-coloring, it is possible to count
the number of vertices from $C_j$ at~distance~$2$ from any vertex
from~$C_i$. If the quotient matrix of the perfect $2$-coloring is
equal to $\begin{pmatrix}
a & b\\
c & d
\end{pmatrix}
$, then there hold $(M{\mathbf 1}_{C_1},{\mathbf 1}_{C_1})=a|C_1|$,
$(M_2{\mathbf{1}}_{C_1},{\mathbf{1}}_{C_1})=(p_2(a^2+bc)+p_1a+p_0)|C_1|$,
and $|C_1|=\frac{cn}{b+c}$. Let
$\beta=\sigma_2(C_1)=p_2(a^2+bc)+p_1a+p_0$. It is clear that
$\frac{c}{b+c}=\frac{bc}{b^2+bc}=\frac{\beta-p(a)}{p_2b^2+\beta-p(a)}$.
Thus we  obtain that $|C|=
\frac{(\beta-p(a))n}{p_2(r_p-a)^2+\beta-p(a)} $ for any perfect
$2$-coloring ${\mathbf{1}}_C$ of  $G$.

Next we prove that a fixed $\sigma(C)$ and a bounded $\sigma_2(C)$
provide an upper bound   for the cardinality of $C$. Moreover, if
this upper bound is reached on $C$ then the partition
$\{C,V\setminus C\}$ is  equitable.

\begin{theorem}\label{thEPPC}Let $G$ be an amply $r$-regular graph with polynomial
$p$ and let $C\subset V$. If  $\sigma(C)=a$ and $\beta=\sigma_2(C)$
then $|C|\leq \frac{(\beta-p(a))n}{p_2(r-a)^2+\beta-p(a)} $.
Moreover, if $|C|= \frac{(\beta-p(a))n}{p_2(r-a)^2+\beta-p(a)} $
then
 ${\mathbf{1}}_C$ is a perfect $2$-coloring of~$G$.\end{theorem}

\bproof Without loss of generality, we suppose that  $G$ is
connected. In the other case,  we can prove the theorem separately
for each component of connectivity. Consider ${ {\mathbf{1}}}_{C}$
as a linear combination of eigenfunctions of $G$. It holds
${\mathbf{1}}_{C}=\sum_i\alpha_i\phi_i$, where $\phi_i$ is an
eigenfunction of~$M$ with   eigenvalue~$\lambda_i$. Without loss of
generality, we assume  $\|\phi_i\|_2=1$ for all~$i$. The
eigenfunction with   eigenvalue~$r$ is equal to
$\phi_0=\mathbf{1}_V/\sqrt{n}$.  From $({
{\mathbf{1}}}_{C},{\mathbf{1}}_{C})= \sum_i\alpha^2_i$ we derive
$$
{\rm (I)} \qquad \sum\limits_{i\neq 0}\alpha^2_i=
({\mathbf{1}}_{C},{ {\mathbf{1}}}_{C})-\alpha^2_0=|C|- \varrho|C|,
$$
 where $\varrho=|C|/n$. From $(M{\mathbf{1}}_{C},{\mathbf{1}}_{C})=
 a|C|$ and $
(M{\mathbf{1}}_{C},{\mathbf{1}}_{C})=\sum_i\alpha^2_i\lambda_i $, it
follows
$$
{\rm (II)} \quad r\varrho|C|+\sum\limits_{i\neq
0}\alpha^2_i\lambda_i= a|C|.$$ From (\ref{e1.2}) and  the hypothesis
of the theorem we obtain
$$(M^2{\mathbf{1}}_{C},{\mathbf{1}}_{C})=
\frac{1}{p_2}((M_2-p_1M-p_0I){\mathbf{1}}_{C},{\mathbf{1}}_{C})=
\frac{|C|}{p_2}(\beta-p_1a-p_0).  $$ Hence,
$$
{\rm (III)} \quad r^2\varrho|C|+\sum\limits_{i\neq
0}\alpha^2_i\lambda^2_i= \frac{|C|}{p_2}(\beta-p_1a-p_0).
$$

Let us combine the left and right parts of equations (I)--(III)
according to the following formula ${\rm (III)}-2\theta{\rm
(II)}+\theta^2{\rm (I)}$.  Then for any $\theta \in \mathbb{R}$ we
obtain the  inequalities
$$r^2\varrho|C|-2 r\varrho\theta|C|+\sum\limits_{i\neq
0}\alpha^2_i(\lambda_i-\theta)^2=
\frac{|C|}{p_2}(\beta-p_1a-p_0)-2a\theta|C|+|C|(1-\varrho)\theta^2,$$
\begin{equation}\label{eqEx8}
r^2\varrho-2 r\varrho\theta\leq
\frac{1}{p_2}(\beta-p_1a-p_0)-2a\theta+(1-\varrho)\theta^2,
\end{equation}
$$\varrho\leq\frac{\frac{1}{p_2}(\beta-p(a))+(a-\theta)^2}{(r-\theta)^2}.$$
Let $\theta=a-\frac{\beta-p(a)}{p_2(r-a)}$. Then we conclude that
$$\varrho\leq\frac{(a-\theta)(r-a)+(a-\theta)^2}{(r-a+a-\theta)^2}=\frac{a-\theta}{r-\theta}=   \frac{\beta-p(a)}{p_2(r-a)^2+\beta-p(a)}.$$

It is clear that (\ref{eqEx8}) holds with equality if and only if
$f=\phi+\alpha_0\varphi_0$, where $\phi$ is an eigenfunction with
eigenvalue $\theta$. By Proposition \ref{corol5} we obtain that
${\mathbf{1}}_{C}$ is a perfect $2$-coloring in this case. \proofend

 For $a=0$ and $\beta=0$, the new bound coincides
 with the
Hamming bound $|C|\leq \frac{n}{r+1} $. If $C$ is an independent
set, then $a=0$, $p(0)=p_0=-rp_2$ and $\varrho\leq
1/(1+\frac{p_2r^2}{\beta+p_2r})$. This bound and the Hoffman bound
are reached simultaneously on perfect $2$-colorings with the minimum
eigenvalue.  In this case, it holds $\sigma_2(C)=
-p_2r(\lambda_{\mathrm{min}}+1)$. By Corollary~\ref{corExt},  we
have $\sigma_2(C)\leq -p_2r(\lambda_{\mathrm{min}}+1)$ for any
independent set~$C$. Consequently, meaningfully consider only the
case of  $\beta< -p_2r(\lambda_{\mathrm{min}}+1)$.

\begin{corollary}\label{corExt2}
Let $C$ be an independent set  in an amply regular graph $G$ with
polynomial~$p$. If $\beta=\sigma_2(C)<
-p_2r(\lambda_{\mathrm{min}}+1)$, then the bound $\frac{|C|}{n}\leq
1/(1+\frac{p_2r^2}{\beta+p_2r})$ is the better than the Hoffman
bound.
\end{corollary}
\bproof
$1/(1+\frac{p_2r^2}{\beta+p_2r})<1/(1+\frac{p_2r^2}{-p_2r\lambda_{\mathrm{min}}})=\frac
{-\lambda_{\mathrm{min}}}{r-\lambda_{\mathrm{min}}}$. \proofend

\bigskip

{\bf Funding.} The work was funded by the Russian Science Foundation
(grant No 22-11-00266).


\begin{thebibliography}{99}

\bibitem{AlonCh}
Alon, N.,  Chung, F.R.K.: Explicit construction of linear sized
tolerant networks. In: Proceedings of the First Japan Conference on
Graph Theory and Applications, Hakone, 1986,  15--19 (1988)

\bibitem{Avg}
Avgustinovich, S.V.: private communication.


\bibitem{Bier}
 Bierbrauer, J.:  Bounds on orthogonal arrays and resilient
 functions.
Journal of Combinatorial Designs,  3 179--183 (1995)

\bibitem{BCN}
  Brouwer, A.E.,  Cohen, A.M.,  Neumaier, A.: Distance-regular graphs.
Ergebnisse der Mathematik und ihrer Grenzgebiete (3) [Results in
Mathematics and Related Areas (3)], 18. Springer-Verlag, Berlin
(1989)







\bibitem{Delsarte}
Delsarte, P.: An algebraic approach to association schemes of coding
theory, volume 10 of Philips Res. Rep., Supplement. N.V. Philips'
Gloeilampenfabrieken, Eindhoven, (1973)

\bibitem{Dev}
Devriendt, K.,    Van Mieghem, P.:  Tighter spectral bounds for the
cut size, based on Laplacian eigenvectors.  Linear Algebra Appl. 572
 68--91 (2019)



\bibitem{FdF2}
Fon-Der-Flaass, D.G.:  A bound on correlation immunity. Siberian
Electronic Mathematical Reports, 4 133--135 (2007)



\bibitem{Friedman}
Friedman, J.:  On the bit extraction problem. In: Proceedings of
33rd IEEE Symposium on Foundations of Computer Science, 314--319
(1992)


\bibitem{Golubev}
Golubev, K.: Graphical designs and extremal combinatorics. Linear
Algebra Appl., 604  490--506 (2020)


\bibitem{Haemers}
 Haemers, W.H.:  Eigenvalue techniques in design and graph theory.
Dissertation, Technische Hogeschool Eindhoven, Eindhoven, 1979.
Mathematical Centre Tracts, 121. Mathematisch Centrum, Amsterdam,
(1980)

\bibitem{Hoff}
  Hoffman, A.J.: On eigenvalues and colorings of graphs, Graph Theory
and its Applications. In: Proceedings of Advanced Sem., Math.
Research Center, Univ. of Wisconsin, Madison, Wis.,   79--91 (1970)




\bibitem{Krotov12}
Krotov, D.S.:   On the binary codes with parameters of
triply-shortened 1-perfect codes. Designs, Codes and Cryptography,
64(3)  275--283 (2012)

\bibitem{Krotov1536}
Krotov, D.S.: On the OA(1536,13,2,7) and related orthogonal arrays.
Discrete Math., 343(2):111659 1--11 (2020) 


\bibitem{CRC}
Krotov, D.S, Potapov, V.N.: Completely regular codes and equitable
partitions. In: Shi, M.,
 Sole, P. (eds.) Completely regular codes in distance regular graphs.  Chapman \& Hall/CRC Monographs and Research Notes
in Mathematics, (2025)

\bibitem{Martin}
Martin, W.J.: Completely regular subsets. Thesis (Ph.D.)-University
of Waterloo (Canada)  ISBN: 978-0315-72488-4, ProQuest LLC (1992)


\bibitem{OPPh}
Ostergard, P.R.J.,  Pottonen, O.,  Phelps, K.T.: The perfect binary
one-error-correcting codes of length 15: Part II-properties. IEEE
Trans. Inform. Theory, 56(6)  2571--2582 (2010)

\bibitem{Pot12}
  Potapov, V.N.: On perfect $2$-colorings of the $q$-ary $n$-cube. Discrete
Math., 312(6)  1269--1272 (2012)

\bibitem{PotAv}
 Potapov, V.N.,  Avgustinovich, S.V.:  Combinatorial designs,
difference sets, and bent functions as perfect colorings of graphs
and multigraphs. Siberian Mathematical Journal, 61(5)   867--877 (2020) 

\bibitem{ATar}
 Taranenko, A.A.: Algebraic properties of perfect structures.
 Linear Algebra Appl. 607  286--306 (2020)


\bibitem {Val24}
 Valyuzhenich, A.: An upper bound on the number of relevant
variables for Boolean functions on the Hamming graph.
https://doi.org/10.48550/arXiv.2404.10418 Accessed 16 April 2024



\end{thebibliography}
\end{document}